\documentclass{amsart}
\usepackage{amsmath,amsthm,latexsym,amssymb,hyperref, amsfonts}
\usepackage{fullpage}
\usepackage{stmaryrd}
\usepackage{eucal}
\usepackage{mathrsfs}
\usepackage[all]{xy}
\usepackage{color}
\usepackage{tikz}
\usepackage{url}
\usepackage{tensor}
\usepackage{comment}
\usepackage{tikz-cd}
\usepackage{enumerate}
\usepackage{enumitem}
\usepackage{mathrsfs}
\usepackage{dutchcal}
\usetikzlibrary{matrix,arrows,decorations.pathmorphing}
\title[Rigidification of Connective Comodules]{Rigidification of Connective Comodules}

\author[Maximilien P\'eroux]{Maximilien P\'eroux}

\theoremstyle{definition}
\newtheorem{defi}{Definition}[section]

\numberwithin{equation}{section}

\theoremstyle{plain}
\newtheorem{thm}[defi]{Theorem}
\newtheorem{prop}[defi]{Proposition}
\newtheorem{lem}[defi]{Lemma}
\newtheorem{cor}[defi]{Corollary}

\renewcommand{\o}{\otimes}

\newcommand{\I}{\mathbb I}

\newcommand{\Ext}{\mathsf{Ext}_C}

\newcommand{\cotor}{\mathsf{CoTor}_C}

\newcommand{\id}{\mathsf{id}}

\newcommand{\D}{\mathsf{D}}
\newcommand{\op}{^\mathsf{op}}
\newcommand{\M}{\mathsf{M}}

\newcommand{\DDk}{\mathcal{D}^{\geq 0}(\k)}

\newcommand{\ccobars}[1]{\Omega^\bullet (#1, C, C)}
\newcommand{\ccobar}[1]{\Omega(#1, C, C)}

\newcommand{\cobar}[2]{\Omega(#1, C, #2)}

\newcommand{\ccotens}{{{\square}_C}}
\newcommand{\cotens}{\square}
\newcommand{\dcotens}{\widehat{\square}_C}

\newcommand{\N}{\mathscr{N}}

\newcommand{\A}{\mathsf{A}}

\newcommand{\holim}{\mathsf{holim}}
\renewcommand{\lim}{\mathsf{lim}}
\newcommand{\colim}{\mathsf{colim}}

\newcommand{\smodk}{\mathsf{sMod}_\k}

\newcommand{\Hom}{\mathsf{Hom}}
\newcommand{\Homc}{\mathsf{Hom}_C}

\newcommand{\pull}{\arrow[dr, phantom, "\lrcorner", very near start]}

\renewcommand{\k}{\Bbbk}
\newcommand{\ch}{\mathsf{Ch}}

\newcommand{\chfo}{{\ch_\k}}

\newcommand{\comod}{\mathsf{CoMod}}

\newcommand{\W}{\mathsf{W}}

\newcommand{\Path}{s^{-1}\mathsf{Cone}}

\newcommand{\comodinf}{\mathcal{CoMod}}
\newcommand{\Dinf}{\mathcal{D}}

\newcommand{\Cinf}{\mathcal{C}}

\newcommand{\modinf}{\mathcal{Mod}}

\tikzset{
    labl/.style={anchor=south, rotate=-32, inner sep=.9mm}
}

\tikzset{
    labll/.style={anchor=south, rotate=32, inner sep=.9mm}
}

\begin{document}

 \address{Department of Mathematics, Michigan State University, 619 Red Cedar Road, East Lansing, MI 48824, USA}
\email{peroux@msu.edu}

    \makeatletter
\@namedef{subjclassname@2020}{\textup{2020} Mathematics Subject Classification}
\makeatother

   

\begin{abstract} 
Let $\k$ be a commutative ring with global dimension zero.
We show that we can rigidify homotopy coherent comodules in connective modules over the Eilenberg-Mac Lane spectrum of $\k$. That is, the $\infty$-category of homotopy coherent comodules is represented by a model category of strict comodules in non-negative chain complexes over $\k$. These comodules are over a coalgebra that is strictly coassociative and simply connected.
The rigidification result allows us to derive the notion of cotensor product of comodules and endows the $\infty$-category of comodules with a symmetric monoidal structure via the two-sided cobar resolution.
\end{abstract}

\maketitle
 

\section{Introduction}

\subsection*{Rigidifying multiplications}
It has long been understood that one can derive the usual tensor product of modules over a commutative ring $R$ to obtain a monoidal structure in the derived category of $R$-modules.
The two-sided bar resolution $B(M,R, N)$ provides an explicit model for the derived tensor product of $R$-modules $M$ and $N$. The homology of $B(M,R, N)$ can be computed as $\mathsf{Tor}_*^R(M, N)$.
More generally, symmetric monoidal model categories provide a symmetric monoidal structure on their associated homotopy categories given by the derived tensor product, see \cite[4.3.2]{hovey}. 
Thanks to the monoidal structure on the homotopy categories, one can consider algebraic structures such as rings, algebras or modules that are associative and unital up to homotopy. However, higher homotopy coherences capture more interesting data. For instance a grouplike $\mathbb{E}_n$-algebra in spaces is an $n$-fold loop space, see \cite{geom}.

In this realization comes an obstacle: how does one define such an algebraic structure? One would need not only to specify a multiplication and unit, but also all the homotopy coherences. 
Therefore there has been a growing interest to describe all the data of homotopy coherences. 
The work of \cite{EKMM, SS, MMSS, MM, convenient} shows that, in stable homotopy theory, we can conveniently choose a model category so that the homotopy coherence can be represented by a strictly associative, commutative and unital algebraic object.
In other words one can rigidify the multiplication so that specifying the data amounts only to define the multiplication and unit.
Higher categories expand on this approach: any presentably symmetric monoidal $\infty$-category is the Dwyer-Kan localization of a combinatorial symmetric monoidal model category, see \cite{NSpres}. Moreover, given some assumptions on the combinatorial symmetric monoidal model category, homotopy coherent associative or commutative algebras can be rigidified and modeled by strictly associative and unital algebras. See for instance \cite[4.1.8.4, 4.3.3.17, 4.5.4.7]{lurie1} or \cite{hinrectification}.

\subsection*{The rise of coalgebraic structures}
Valuable geometric information can be stored by considering homotopy coherent coalgebraic structures. 
For instance, it was shown in \cite[1.3]{corecognition} that a $2$-connected space is equivalent to a suspension if and only if it is an $\mathbb{A}_\infty$-cogroup in pointed spaces. Extending the results in \cite{mandello}, the work of \cite{yuan} uses homotopy coherent coalgebras and Frobenius actions to show that the homotopy of a simply connected $p$-complete finite space is entirely determined by some $\mathbb{E}_\infty$-coalgebra on the chain complexes. Moreover, $\mathbb{A}_\infty$-ring spectra are enriched in $\mathbb{A}_\infty$-coalgebras \cite{coalgenr}.

As in algebras, we need to construct homotopy coherent coalgebraic structures. In particular, we would want to rigidify homotopy coherent comultiplications and coactions. 
However, we face multiple obstacles. Firstly, a common assumption for algebras and modules is that the monoidal product is closed and preserves colimits. So in order to dualize the results for coalgebras and comodules, we would need a monoidal product that preserves limits, which in practice is never the case.
Moreover, it is difficult to endow a model structure on the categories for coalgebras and comodules, see \cite{left1, left2, left3}. In the cases where a model structure does exist, the author has shown that they may not model the correct $\infty$-category. For instance, strictly coassociative coalgebras in symmetric spectra cannot model $\mathbb{A}_\infty$-coalgebras in spectra, see \cite{perouxshipley} and \cite{coalginDK}. This raises an issue as we would want to give a monoidal structure on homotopy coherent comodules. If one wanted to mimic the monoidal structure on modules given by the derived relative tensor product, one could be tempted to dualize the result of \cite[4.5.2.1]{lurie1} but it would require the underlying tensor product to preserve totalizations. In the underived setting, the cotensor product of comodules only provides a monoidal structure if considered in a flat setting. Attempts to consider a derived cotensor product as a two-sided cobar construction was considered in \cite[4.3]{cobarr}. In \cite{chiral}, the authors forced the tensor product to commute with totalizations but the setting is limiting in practice.

\subsection*{Results}
Our main result here shows that we can rigidify certain homotopy coherent coactions of comodules over a coalgebra. Let $\k$ be a commutative ring of global dimension zero, i.e. a finite product of fields, and denote $H\k$ its Eilenberg-Mac Lane spectrum. We consider $\Dinf^{\geq 0}(\k)$ the symmetric monoidal $\infty$-category of \emph{connective} modules over $H\k$. It is represented by the symmetric monoidal model category of non-negative chain complexes over $\k$, denoted $\chfo$, in which weak equivalences are quasi-isomorphisms. In other words, the Dwyer-Kan localization of $\chfo$ is  $\Dinf^{\geq 0}(\k)$.
If $C$ is a \emph{simply connected} coalgebra in $\chfo$, then we show that the $\infty$-category of (homotopy coherent) comodules in $\Dinf^{\geq 0}(\k)$ over $C$ is represented by a model category of (strictly coassociative) differential graded comodules over $C$ in which weak equivalences are quasi-isomorphisms of $C$-comodules.

\begin{thm}[{Theorem \ref{main thm}}]\label{THEOREM 1}
Let $C$ be a simply connected differential graded coalgebra over $\k$. Then right $C$-comodules in the $\infty$-category $\DDk$ of connective $H\k$-modules are quasi-isomorphic to strictly coassociative differential graded connective comodules over $C$.
\end{thm}

One of the main application of the above result is that it allows us to define a symmetric monoidal structure on the $\infty$-category of $C$-comodules in $\DDk$. It is called the \emph{derived cotensor product} and is given by a two-sided cobar resolution.
As recalled above, simply dualizing the results for modules seen in  \cite[4.5.2.1]{lurie1} is not possible, as it would require totalizations to commute with the monoidal product, which is \emph{not} the case in $\DDk$.
Theorem \ref{THEOREM 1} allows us to bypass this difficulty and use an explicit model for the derived cotensor product.


\begin{thm}[{Theorem \ref{thm: derived cotensor product}}]\label{THEOREM 2}
Let $C$ be a simply connected cocommutative differential graded coalgebra over $\k$. The $\infty$-category $\comodinf_C(\DDk)$ of $C$-comodules in $\DDk$ is endowed with a symmetric monoidal structure induced by $\dcotens$, the derived cotensor of comodules  which is given by the two-sided cobar resolution: $M\dcotens N \simeq \cobar{M}{N}$. Moreover:
\begin{itemize}
\item \textup{Theorem \ref{thm: EMSS}:} there is an Eilenberg-Moore spectral sequence computing $H_*(X\dcotens Y)$ with an $E^2$-page given by $\mathsf{CoTor}_{H_*(C)}(H_*(X), H_*(Y))$;
\item \textup{Proposition \ref{prop: change of coalgebras}:} if $C$ and $D$ are quasi-isomorphic simply connected cocommutative differential graded coalgebras, then $\comodinf_C(\DDk)$ and $\comodinf_D(\DDk)$ are equivalent as symmetric mono\-idal $\infty$-categories.
\end{itemize}
\end{thm}

In \cite{cothhshadow}, the above results are applied to provide a bicategorical trace on coHochschild homology. Theorems \ref{THEOREM 1}  and \ref{THEOREM 2} are crucial to define a bicategory on derived connective bicomodules in chain complexes.
Although (homotopy) colimits of modules are well understood, understanding (homotopy) limits of comodules is much more challenging. A key observation in our proof is that we can determine homotopy limits of certain explicit cosimplicial objects in comodules.

\subsubsection*{Simplicial comodules} The $\infty$-category $\DDk$ of connective $H\k$-modules is also represented by the symmetric monoidal category of simplicial $\k$-modules $\smodk$ in which the weak equivalences are weak homotopy equivalences. 
All arguments in this paper remain valid if we replace $\chfo$ by $\smodk$. In particular Theorems \ref{THEOREM 1} and \ref{THEOREM 2} have a simplicial version. However $\smodk$ and $\chfo$ are not equivalent as monoidal categories, but they are after we derive their tensor products. We show in \cite{postnikovperoux} there is a Dold-Kan correspondence between simplicial comodules and differential graded connective comodules which induces an equivalence of their derived cotensor products.
Therefore the induced derived cotensor product of comodules in $\DDk$ in Theorem \ref{THEOREM 2} is independent from choosing between $\smodk$ or $\chfo$.

\subsubsection*{A rigidification for connective bicomodules} For the sake of simplicity, we focus in this paper on right comodules over a coalgebra, but the above theorems remain true for left comodules. In fact, given $C$ and $D$ simply connected differential graded coalgebras, we can obtain a rigidification result as in Theorem \ref{THEOREM 1} for $(C, D)$-bicomodules, as they are equivalent to left $(C\otimes D^\mathsf{op})$-comodules. In particular, we can drop the cocommutativity requirement of Theorem \ref{THEOREM 2} and obtain a derived cotensor product for bicomodules over $C$ which provides a monoidal (but not symmetric) structure. We provide further details in \cite{cothhshadow}.
 
\subsubsection*{On the global dimension zero condition} Throughout this paper, we work exclusively with $\k$ a commutative ring of global dimension zero, i.e. a finite product of fields. There are several reasons this condition is imposed. First, it induces that every object is cofibrant and fibrant in $\chfo$, and the projective and injective model structures are equal. In particular, the model structure on comodules is left-induced from a nice monoidal model category. In \cite{left2}, it was shown that the model structures for comodules are left-induced from injective model structures which are in general not monoidal model categories. This creates several issues to understand the induced homotopy theory on comodules. Moreover, as  every module is flat, the tensor product preserves finite limits. This allows us to understand finite limits of comodules, which is essential for Theorem \ref{THEOREM 1}. Finally,  the cotensor product of comodules is \emph{not} a comodule without the flat setting, hence Theorem \ref{THEOREM 2} cannot be true. Notice though that in \cite[6.4.7]{phd}, we proved an unbounded version of Theorem \ref{THEOREM 1}, where $\k$ is replaced by any commutative ring and $C$ is a differential graded coalgebra that is perfect as a chain complex.

\subsubsection*{On the simply connected condition}
A reader familiar with the notion of Koszul duality might not be surprised of the requirement of our coalgebra $C$ to be simply connected, as there is an equivalence between comodules over $C$ with certain modules over the cobar construction $\Omega (\k, C, \k)$, see \cite{leo}. However our main results would not follow through this duality as the relation is not well established yet in the $\infty$-categorical setting, and even so, it is in general not a monoidal equivalence.

\subsection*{Outline} 
We introduce the model structure on the  category of comodules in Section \ref{Section: cobar}, and then provide an explicit fibrant replacement given by the two-sided cobar resolution in Theorem \ref{thm: fibrant replacement as cobar}. 
We prove Theorem \ref{THEOREM 1} in Section \ref{Section: rigidifying}, and we apply this result to define the derived cotensor product of Theorem \ref{THEOREM 2} in Section \ref{Section: derived cotensor}.

\subsection*{Acknowledgment}
The results here are part of my PhD thesis \cite{phd}, and as such, I would like to express my gratitude to my advisor Brooke Shipley for her help and guidance throughout the years. I am also grateful to Pete Bousfield who read earlier versions of this paper and suggested the result Lemma \ref{lem: Bousfield}. Special thanks to initial conversations with Ben Antieau, Jonathan Beardsley and Kathryn Hess that sparked some ideas in this paper. Thanks to David Chan, Sarah Klanderman and Elden Elmanto for their comments.

\subsection*{Notation}
We begin by setting notation that we  use throughout and recalling some elementary notions of the theory of $\infty$-categories, following \cite{htt, lurie1}. The notions of \emph{symmetric monoidal $\infty$-categories} and \emph{$\infty$-operads} are defined respectively in \cite[2.0.0.7, 2.1.1.10]{lurie1}. By an \emph{ordinary} category, we shall refer to a category in the usual sense of the term. By an \emph{$\infty$-category}, we mean a quasicategory as in \cite{htt}. An ordinary category shall be denoted with bold letters $\mathsf{C}$, $\mathsf{A}$, etc, while an $\infty$-category shall be denoted with cursive letters $\mathcal{C}$, $\mathcal{A}$, etc.

\begin{enumerate}

\item The letter $\k$ shall always denote a commutative ring of global dimension zero, i.e. a finite product of fields. 
In the literature, such rings are referred to as \emph{commutative semisimple rings}.

\item Let $\chfo$ be the category of non-negative chain complexes of $\k$-modules (graded homologically). The category is endowed with a symmetric monoidal structure. The tensor product of two chain complexes $X$ and $Y$ is defined by:
\(
(X\otimes Y)_n=\bigoplus_{i+j=n} X_i\otimes_\k Y_j,
\)
with differential given on homogeneous elements by:
\(
d(x\otimes y)=dx \otimes y + (-1)^{\vert x \vert}x\otimes dy.
\)
We denote the tensor simply as $\otimes$. The monoidal unit is denoted $\k$, which is the chain complex $\k$ concentrated in degree zero.

\item Given $V$ a graded $\k$-module, we define
    \( T(V)=\bigoplus_{n\geq 0} V^{\otimes n}.\)
    Elements in the summands are denoted $v_1 \vert \cdots \vert v_n$, where $v_i\in V$.
\item Let $s^{-1}$ denote the desuspension functor on graded $\k$-modules where, for $V=\bigoplus_{i\in \mathbb{Z}} V_i$, we define $(s^{-1} V)_i=V_{i+1}$.
    Given a homogeneous element $v$ in $V$, we write $s^{-1} v$ for the element in $s^{-1}V$.


\end{enumerate}

\section{The two-sided cobar resolution}\label{Section: cobar}

We begin by first introducing some elementary coalgebraic notions. By a \emph{coalgebra} $(C, \Delta, \varepsilon)$ in $\chfo$ we mean a chain complex $C$ together with chain maps $\Delta\colon C\rightarrow C\otimes C$ and $\varepsilon\colon C\rightarrow \k$ such that it is coassociative, i.e. $(\id_C\otimes \Delta)\circ \Delta=(\Delta \otimes \id_C)\circ \Delta)$, and counital, i.e. $(\id_C\otimes \varepsilon)\circ \Delta = \id_C= (\varepsilon\otimes \id_C)\circ \Delta$. It is \emph{cocommutative} if in addition $\tau\circ \Delta=\Delta$ where $\tau$ is the symmetric isomorphism in the monoidal structure of $\chfo$. We say the coalgebra $C$ is \emph{simply connected} if $C_0=\k$ and $C_1=0$. 
We shall occasionally use the Sweedler notation for the comultiplication and write simply $c_{(1)}\otimes c_{(2)}$ for $\Delta(c)=\sum_{i} {c_{(1)}}_i\otimes {c_{(2)}}_i\in (C\otimes C)_n$ for any $c\in C_n$.

Given a coalgebra $(C, \Delta, \varepsilon)$, a \emph{right $C$-comodule} $(M, \rho)$ is a chain complex $M$ together with a chain map $\rho\colon M\rightarrow M\otimes C$ that is coassociative, i.e. $(\id_M \otimes \Delta) \circ \rho=(\rho \otimes \id_C)\circ \rho$, and counital, i.e. $(\id_M\otimes \varepsilon)\circ \rho=\id_M$.
A (right) $C$-colinear map $f\colon (M,\rho)\rightarrow (M', \rho')$ is a chain map $f\colon M\rightarrow M'$ such that $\rho'\circ f= (f\otimes \id_C)\circ \rho$. 
We denote by $\comod_C$ the induced category of right $C$-comodules. Left $C$-comodules are defined completely analogously.  If $C$ is cocommutative, then  left and right $C$-comodules  are equivalent and we will simply refer to them as $C$-comodules. We shall occasionally use the Sweedler notation for the coaction and write simply $x_{(0)}\otimes x_{(1)}$ for $\rho(x)=\sum_{i} {x_{(0)}}_i\otimes {x_{(1)}}_i\in (M\otimes C)_n$ for any $x\in M_n$.

One can view $\comod_C$ as the category of coalgebras over the comonad $-\otimes C\colon \chfo \rightarrow \chfo$. Thus by \cite[2.78, 2.j]{Adamek-Rosicky}, the category $\comod_C$ is locally presentable. 
Moreover, the forgetful functor $U\colon \comod_C\rightarrow \chfo$ is left adjoint to the cofree functor $-\otimes C\colon \chfo\rightarrow \comod_C$. Therefore colimits in $\comod_C$ are computed in $\chfo$. Since $C$ is automatically flat in $\chfo$, we get that the comonad $-\otimes C\colon \chfo\rightarrow \chfo$ preserves finite limits, and thus finite limits in $\comod_C$ are also determined in $\chfo$, and the category $\comod_C$ is abelian. 
If $\D$ is a small category (with infinitely many objects) and $F$ is a $\D$-diagram in $\comod_C$, we denote by $\lim^C_\D F$ the limit in $\comod_C$ and by $\lim_\D F$ the limit of $U\circ F$ in $\chfo$.

\begin{defi}
Let $(M, \rho)$ and $(N, \lambda)$ be respectively right and left comodules over a coalgebra $C$ in $\chfo$.
Define the \emph{cotensor product $M\square_C N$ of $M$ and $N$} to be the equalizer in $\chfo$: 
\[
\begin{tikzpicture}[baseline= (a).base]
\node[scale=1] (a) at (1,1){
\begin{tikzcd}
M\ccotens N \ar{r} & M\otimes N \ar[shift left]{r}{\rho\otimes \id_N}\ar[shift right]{r}[swap]{\id_M\otimes \lambda} & [1em] M\otimes C \otimes N.
\end{tikzcd}
};
\end{tikzpicture}
\]
If $C$ is cocommutative, as $-\otimes C\colon \chfo\rightarrow \chfo$ preserves equalizers since $C$ is automatically flat in $\chfo$, we get that $M\square_C N$ is a $C$-comodule via the coaction on $N$ (or equivalently on $M$):
\[
\begin{tikzpicture}[baseline= (a).base]
\node[scale=1] (a) at (1,1){
\begin{tikzcd}
M\ccotens N \ar[dashed]{d} \ar{r} & M\otimes N \ar{d} \ar[shift left]{r}\ar[shift right]{r} & M\otimes C \otimes N\ar{d}\\
(M \ccotens N) \otimes C \ar{r} & M\otimes N \otimes C \ar[shift left]{r}\ar[shift right]{r} & M\otimes C \otimes N \otimes C.
\end{tikzcd}
};
\end{tikzpicture}
\]In fact, it is elementary to check that the cotensor product extends to a symmetric monoidal structure on $\comod_C$ for which $C$ is the monoidal unit.
\end{defi}

There is a model structure on $\chfo$ in which weak equivalences are quasi-isomorphisms, cofibrations are monomorphisms, and fibrations are positive levelwise epimorphisms (\cite{quillen}). It is a combinatorial symmetric monoidal model category, see \cite{monmodSS}. Every object is cofibrant and fibrant.
A \emph{quasi-isomorphism of (right) $C$-comodules} is a $C$-colinear map such that it induces an isomorphism on the homologies as a chain map.

\begin{prop}\label{prop: comodules have a nice model cat}
Let $C$ be a coalgebra in $\chfo$.
Then the category of right comodules admits a combinatorial model structure left-induced from the forgetful-cofree adjunction
\[
\begin{tikzcd}[column sep= large]
U\colon\comod_C\ar[shift left=2]{r}{}[swap]{\perp} & \chfo\colon -\o C. \ar[shift left=2]{l}{}
\end{tikzcd}
\]
In particular, $U$ preserves and reflects cofibrations and weak equivalences. Therefore the weak equivalences in $\comod_C$ are precisely the quasi-isomorphisms of right $C$-comodules, and the cofibrations are precisely the monomorphisms of right $C$-comodules. Every object is cofibrant. 
\end{prop}

\begin{proof}
One can adapt the arguments from \cite[6.3.7]{left2} for unbounded chain complexes to $\chfo$. The model structure is combinatorial by \cite[2.23]{left1} combined with \cite[3.3.4]{left2}. 
\end{proof}

When $C$ is cocommutative, we will see that fibrant $C$-comodules correspond precisely to ``coflat" comodules, see Proposition \ref{prop: coflat=fibrant}.
We now introduce an explicit fibrant replacement in $\comod_C$. The monad $(-\otimes C) \circ U$ in $\comod_C$ induces a cosimplicial resolution for any object in $\comod_C$. Subsequently, we shall sometimes omit the functor $U$ for convenience.
This leads to the following definition.

\begin{defi}\label{def: cobar cosimplicial}
Let $(C, \Delta, \varepsilon)$ be a coalgebra in $\chfo$.
Let $(M, \rho)$ be a right $C$-comodule. The \emph{cosimplicial cobar construction} $\ccobars{M}$ of $M$ is the coaugmented cosimplicial object in $\comod_C$:
\[
\begin{tikzpicture}[baseline= (a).base]
\node[scale=1] (a) at (1,1){
\begin{tikzcd}
M\ar{r} & M\otimes C \ar[shift left]{r}\ar[shift right]{r} \ar[bend  right, dashed]{l} &  \ar[bend  right, dashed]{l} \ar{l} M\otimes C \otimes C \ar[shift left=2]{r}\ar{r} \ar[shift right=2]{r} & \ar[bend  right, dashed]{l} \ar[shift right]{l} \ar[shift left]{l} \cdots, 
\end{tikzcd}
};
\end{tikzpicture}
\]
defined as follows. For all $n\geq -1$: $\Omega^{n}(M, C, C)=M\otimes C^{\otimes n+1}$.
The zeroth coface map is given by $d^0=\rho\otimes \id_{C^{\otimes n+1}}\colon \Omega^{n}(M, C, C)\rightarrow \Omega^{n+1}(M, C, C)$.
 For $1\leq i \leq n+1$, the $i$-th coface map is given by: \[{ d^i= \id_M \otimes \id_{C^{\otimes i-1}} \otimes \Delta \otimes \id_{C^{\otimes n+1-i}}\colon \Omega^{n}(M, C, C) \rightarrow \Omega^{n+1}(M, C, C)}\]
The $j$-th codegeneracy map is given by:
\(
s^j=\id_M\otimes \id_{C^{\otimes j}} \otimes \varepsilon \otimes \id_{C^{\otimes n+1-j}}\colon \Omega^{n+1}(M, C, C) \rightarrow \Omega^{n}(M, C, C),
\)  for all $0\leq j\leq n$.
It is $U$-split coaugmented in the sense that there exist extra degeneracies:
\[
s^{-1}=\id_M\otimes \id_{C^{\otimes n+1}} \otimes \varepsilon\colon \Omega^{n+1}(M, C, C) \rightarrow \Omega^{n}(M, C, C),
\]
for all $n\geq -1$, which are dashed in the diagram above. These extra degeneracies are \emph{not} $C$-colinear but only chain maps. 
If $N$ is a left $C$-comodule, denote by $\Omega^\bullet(M, C, N)$ the cosimplicial object $\Omega^\bullet(M,C,C)\square_C N$ in $\chfo$ where we have extended $-\square_C N\colon \comod_C\rightarrow \chfo$ on the cosimplicial level, and we refer to it as the \emph{two-sided cosimplicial cobar construction}.
\end{defi}

As the cobar construction is $U$-split this means that for any right $C$-comodule $M$, we have that the natural chain map $M\rightarrow \holim_\Delta \Omega^\bullet(M,C, C)$ is a quasi-isomorphism. 
Notice in general that $\holim_\Delta^C \Omega^\bullet(M,C, C)$ is \emph{not} quasi-isomorphic to $\holim_\Delta \Omega^\bullet(M,C,C)$.
We show however that for $C$ simply connected, then we do obtain an equivalence.

\begin{defi}\label{def: double complex cobar}
When $C$ is simply connected, we obtain the coalgebra splitting $C=\k\oplus  \overline{C}$ where $\overline{C}$ is the kernel of the counit of $C$.
Conormalization provides an equivalence between cosimplicial objects and (non-negative) cochain complexes in $\chfo$, see \cite[8.4.3]{weibel}.
Denote $\overline{\Omega}^\bullet(M,C,N)$ the conormalization of $\Omega^\bullet(M,C,N)$. It is a cochain complex in $\chfo$ where \[\overline{\Omega}^n(M,C,N)=M\otimes \overline{C}^{\otimes n}\otimes N.\]
The resulting double complex $\overline{\Omega}^\bullet(M,C,N)_\bullet$ is bounded when $C$ is simply connected as $\overline{\Omega}^q(M,C,N)_p=0$ for $0\leq p \leq 2q-1$,
as $\overline{C}_0=\overline{C}_1=0$.
Therefore the direct-sum and direct-product totalizations of the double complex $\overline{\Omega}^\bullet(M,C,N)_\bullet$ are equal, let us denote $\Omega(M,C,N)$ its totalization.
Explicitly, $\Omega(M,C,N)$ is the chain complex
\(
(M\otimes T(s^{-1}\overline{C}) \otimes N, \delta),
\)
where up to Koszul sign, the differential $\delta$ is defined as
\[
\begin{tikzpicture}[baseline= (a).base]
\node[scale=0.85] (a) at (1,1){
$\delta(x \otimes s^{-1} c_1 \vert \cdots \vert s^{-1}c_n \otimes y)  =  dx \otimes s^{-1}c_1 \vert \cdots \vert s^{-1}c_n \otimes y 
 \pm x_{(0)}\otimes s^{-1} x_{(1)}\vert s^{-1} c_1 \vert \cdots \vert s^{-1} c_n \otimes y\pm x\otimes d_\Omega\Big(s^{-1} c_1 \vert \cdots \vert s^{-1} c_n\Big) \otimes y$
};
\node[scale=0.85] (b) at (1,0.5){$\pm x\otimes s^{-1}c_1 \vert \cdots \vert s^{-1}c_n \otimes dy
 \pm x\otimes s^{-1} c_1 \vert \cdots \vert s^{-1} c_n \vert s^{-1}y_{(1)} \otimes y_{(0)},$};
\end{tikzpicture}
\]
where $d$ denotes either the differential on $M$ or $N$, and $x\mapsto x_{(0)}\otimes x_{(1)}$ denotes the coaction of $M\rightarrow M\otimes C$ applied to an element $x\in M_i$, and $y\mapsto y_{(1)}\otimes y_{(0)}$ denotes the coaction of $N\rightarrow C\otimes N$ applied to an element $y\in N_j$. The differential $d_\Omega$ is defined as:
\[
\begin{tikzpicture}[baseline= (a).base]
\node[scale=0.90] (a) at (1,1){
$d_\Omega(s^{-1} c_1 \vert \cdots \vert s^{-1} c_n) = \sum_{j=1}^n \pm s^{-1}c_1 \vert \cdots \vert s^{-1}(dc_j)\vert \cdots \vert s^{-1}c_n
 + \sum_{j=1}^n \pm s^{-1}c_1 \vert \cdots \vert s^{-1} {c_{j}}_{(1)} \vert s^{-1}{c_j}_{(2)} \vert \cdots \vert s^{-1}c_n$
 };
\end{tikzpicture}
 \]
\end{defi}

When $N=C$, we can give a $C$-comodule structure on the chain complex $\Omega(M,C,C)$ induced by the graded cofree structure $x\otimes c\mapsto x\otimes c_{(1)}\otimes c_{(2)}$.

\begin{thm}\label{thm: fibrant replacement as cobar}
Let $(M, \rho)$ be a right comodule over a simply connected coalgebra $C$ in $\chfo$.
The coaction $\rho\colon M\rightarrow M\otimes C$ factors in $\comod_C$ as
\[
\begin{tikzpicture}[baseline= (a).base]
\node[scale=1] (a) at (1,1){
\begin{tikzcd}[column sep=small]
M \ar[hook]{rr}{\rho} \ar[hook]{dr}{\simeq}[swap]{j}& & M \otimes C\\
& \ccobar{M} \ar[two heads]{ur}[swap]{\widehat{\rho}}
\end{tikzcd}
 };
\end{tikzpicture}
\]
where $j$ is a trivial cofibration and $\widehat{\rho}$ is a fibration in $\comod_C$ specified by $j(x)=x_{(0)}\otimes 1 \otimes x_{(1)}$ and $\widehat{\rho}(x\otimes 1 \otimes c)=x\otimes c$, while $\widehat{\rho}(x\otimes s^{-1}c_1 \vert \dots \vert s^{-1}c_n \otimes c)=0$ for all $n\geq 1$. 
Here $x$ is a homogeneous element of $M$ and $c$ and $c_i$ are homogeneous elements of $C$.
Additionally, we have a quasi-isomorphism $\Omega(M,C,C)\simeq \holim_\Delta^C\Omega^\bullet(M,C,C)$.
Moreover, both $\Omega(M,C,C)$ and $M\otimes C$ are fibrant in $\comod_C$, hence $\Omega(M,C,C)$ is a fibrant replacement of $M$ in $\comod_C$.
\end{thm}

\begin{proof}
It is well-known that  $\holim_\Delta\Omega^\bullet(M,C,C)$ in $\chfo$ is quasi-isomorphic to the (direct-product) totalization of the double chain complex $\overline{\Omega}^\bullet(M,C, C)_\bullet$ (see for instance \cite[4.23]{ulrich}). It is elementary to check that the resulting quasi-isomorphism $M\rightarrow \holim_\Delta  \Omega^\bullet(M,C,C) \simeq \Omega(M,C,C)$ is precisely the map $j$, which is a monomorphism. 
It is also an elementary proof to show that $j$ and $\widehat{\rho}$ are $C$-colinear maps.
We are only left to show that $\widehat{\rho}$ is a fibration and that $\Omega(M,C,C)$ is the homotopy limit in $\comod_C$ of $\Omega^\bullet(M,C,C)$.

We shall omit the signs from the Koszul rule for simplicity.  Given a positive chain complex $(X, d)$, i.e. $X_0=0$, denote $\Path(X)=(X\oplus s^{-1}X, D)$ where $D(x)=d(x)+s^{-1}x$ and $D(s^{-1}x)=-s^{-1}d(x)$, as in \cite[1.5.1]{weibel}. 
Notice that if $N$ is a right $C$-comodule, then so is $\Path(N)$. Denote $e_N\colon \Path(N)\rightarrow N$ the projection, which is $C$-colinear.
Observe that if $f\colon M\rightarrow N$ is a $C$-colinear map between right $C$-comodules, where $N_0=0$, then the pullback of $f$ along $e_N$ is the $C$-comodule $M\oplus s^{-1} N$ with differential $D_f$ where $D_f(x)=dx+s^{-1}f(x)$ and $D_f(s^{-1}y)=-s^{-1}dy$.

We introduce now, for all $n\geq 0$ the positive chain complex $B(n)=s^{-n}(M\otimes \overline{C}^{\otimes n+1})$.
As the quotient $e_n\colon \Path(B(n))\rightarrow B(n)$ is a positive epimorphism, it is a fibration in $\chfo$. 
Therefore the induced cofree map $e_n^C\colon\Path(B(n))\otimes C\cong \Path(B(n)\otimes C)\rightarrow B(n)\otimes C$ is a fibration in $\comod_C$. 
Define the cofree $C$-comodule $E(0)=M\otimes C$ and define $f^{(1)}\colon E(0)\rightarrow B(0)\otimes C$ as $\rho \otimes \id_C + \id_M\otimes\Delta$ where we implicitly composed with the quotient $C\rightarrow \overline{C}$.
The map $f^{(1)}$ is $C$-colinear as it is the sum of colinear maps. Let $E(1)$ be the following pullback in $\comod_C$
\[
\begin{tikzpicture}[baseline= (a).base]
\node[scale=1] (a) at (1,1){
\begin{tikzcd}
E(1) \pull \ar{r} \ar[two heads]{d}[swap]{\widehat{\rho}^{(1)}} & \Path(B(0))\otimes C \ar[two heads]{d}{e_0^C}\\
E(0) \ar{r}[swap]{f^{(1)}} & B(0)\otimes C.
\end{tikzcd}
};
\end{tikzpicture}
\]
Explicitly $E(1)=(M\otimes C) \oplus s^{-1}(M\otimes \overline{C}\otimes C)$ with differential denoted $d^{(1)}_\Omega$ defined on $M\otimes C$ as
\[
d^{(1)}_\Omega=d\otimes \id_C + \id_M\otimes d+s^{-1}(\rho\otimes \id_C + \id_M\otimes \Delta)
\]
while on $M\otimes \overline{C}\otimes C$ as
\[
d^{(1)}_\Omega s^{-1}=-s^{-1}(d\otimes \id_{\overline{C}\otimes C} + \id_M\otimes d \otimes \id_C + \id_{M\otimes \overline{C}} \otimes d).
\]
The resulting map $\widehat{\rho}^{(1)}\colon E(1)\rightarrow E(0)$ is a fibration in $\comod_C$.

By induction, suppose we have the $C$-comodule $E(n)=(M\otimes C) \oplus \bigoplus_{k=1}^n s^{-k}(M\otimes \overline{C}^{\otimes k} \otimes C)$
for all $1\leq n < N$ for some $N>1$, with differential $d^{(n)}_\Omega$ defined so that on the pieces $M\otimes \overline{C}^{\otimes k} \otimes C$ where $0\leq k <n$ we have:
\[
\begin{tikzpicture}[baseline= (a).base]
\node[scale=1] (a) at (1,1){
$\displaystyle d^{(n)}_\Omega s^{-k}=(-1)^{k}s^{-k}\big( d\otimes \id_{\overline{C}^{\otimes k}\otimes C} + \sum_{i=0}^{k-1} \id_{M\otimes \overline{C}^{\otimes i}}\otimes d \otimes \id_{\overline{C}^{k-i-1}\otimes C}+ \id_{M\otimes \overline{C}^{\otimes k}}\otimes d\big)$
};
\node[scale=1] (b) at (1.6,0){
$\displaystyle + s^{-k-1}\big( 
\rho \otimes \id_{\overline{C}^{\otimes k}\otimes C} + \sum_{i=0}^{k-1}\id_{M\otimes \overline{C}^{\otimes i}}\otimes \Delta \otimes \id_{\overline{C}^{k-i-1}\otimes C}+ \id_{M\otimes \overline{C}^{\otimes k}}\otimes \Delta
\big),$};
\end{tikzpicture}
\]
while on the piece $M\otimes \overline{C}^{\otimes n}\otimes C$:
\[
\begin{tikzpicture}[baseline= (a).base]
\node[scale=1] (a) at (1,1){$\displaystyle
d^{(n)}_\Omega s^{-n}=(-1)^{n}s^{-n}\big( d\otimes \id_{\overline{C}^{\otimes n}\otimes C} + \sum_{i=0}^{n-1} \id_{M\otimes \overline{C}^{\otimes i}}\otimes d \otimes \id_{\overline{C}^{n-i-1}\otimes C}+ \id_{M\otimes \overline{C}^{\otimes n}}\otimes d\big).$
};
\end{tikzpicture}
\]
Define $f^{(N)}\colon E(N-1)\rightarrow B(N-1)\otimes C$ as trivial on all terms of the sum of $E(N-1)$ except on $M\otimes \overline{C}^{\otimes N-1}\otimes C$ for which is defined as:
\[
\begin{tikzpicture}[baseline= (a).base]
\node[scale=0.90] (a) at (1,1){
$f^{(N)}s^{-N+1}=s^{-N+1} (\rho\otimes \id_{\overline{C}^{N-1}}\otimes \id_C + \sum_{k=0}^{N-1} \id_M\otimes \id_{\overline{C}^{\otimes k}} \otimes \Delta \otimes \id_{\overline{C}^{\otimes N-k-2}}\otimes \id_C + \id_M\otimes \id_{\overline{C}^{\otimes N-1}} \otimes \Delta )$
};
\end{tikzpicture}
\]
for which we are again using implicitly the quotient $C\rightarrow \overline{C}$.
As before, we can check that $f^{(N)}$ is a chain map that is $C$-colinear.
Let $E(N)$ be the following pullback in $\comod_C$:
\[
\begin{tikzpicture}[baseline= (a).base]
\node[scale=1] (a) at (1,1){
\begin{tikzcd}
E(N) \pull \ar{r} \ar[two heads]{d}[swap]{\widehat{\rho}^{(N)}} & \Path(B(N-1))\otimes C \ar{d}{e_{N-1}^C} \\
E(N-1) \ar{r}[swap]{f^{(N)}} & B(N-1)\otimes C.
\end{tikzcd}
};
\end{tikzpicture}
\]
As above, the $C$-comodule is explicitly given by $E(N)=(M\otimes C) \oplus \bigoplus_{k=1}^N s^{-k}(M\otimes \overline{C}^{\otimes k} \otimes C)$ with differential $d^{(N)}_\Omega$ defined similarly as $d^{(n)}_{\Omega}$.

We have therefore constructed a tower of fibrations in $\comod_C$:
\[
\begin{tikzcd}
\dots \ar[two heads]{r} & E(n) \ar[two heads]{r}{\widehat{\rho}^{(n)}} & E(n-1) \ar{r} & \dots  \ar{r} & E(1) \ar[two heads]{r}{\widehat{\rho}^{(1)}} & E(0)=M\otimes C.
\end{tikzcd}
\]
The composition of the tower $\lim_n^C E(n)\rightarrow M\otimes C$ is a fibration in $\comod_C$, and we argue now that it is precisely the map $\widehat{\rho}$.
For this we notice that $\lim_n^C E(n)\cong \lim_n E(n)$ even though it is an infinite limit. 
Indeed, we can give a right $C$-comodule structure on $\lim_n E(n)$ as we have:
\[
U(E(n))\cong (M\otimes C) \oplus \bigoplus_{k=1}^n M\otimes (s^{-1}\overline{C})^{\otimes k} \otimes C.
\]
Notice that the maps in the tower induce $E(n)_i\cong E(n+1)_i \cong \dots$ as $(M\otimes (s^{-1}\overline{C})^{\otimes n} \otimes C)_i=0$ for all $0\leq i< n$ since $(s^{-1} \overline{C})_0=0$ as $C$ is simply connected.
Therefore the coaction on $\lim_n E(n)$ is a sequence of maps $(\lim_n E(n))_i\rightarrow ((\lim_n E(n)) \otimes C)_i$ that are entirely determined by the coactions $E(n)_i\rightarrow (E(n)\otimes C)_i$ for $n$ large enough.  It is easy to then verify that the natural maps $\lim_n E(n)\rightarrow E(n)$ are $C$-colinear therefore enforcing that $\lim_n^C E(n)\cong \lim_n E(n)$. Thus $\Omega(M,C,C)\cong \lim^C_n E(n)\simeq \holim^C_\Delta \Omega^\bullet(M,C,C)$ and $\widehat{\rho}$ must be a fibration.
\end{proof}

\section{Rigidifying the coaction}\label{Section: rigidifying}

We first make the notion of homotopy coherent coassociative comodules precise using an $\infty$-categorical approach. 
Let $\Cinf$ be a symmetric monoidal $\infty$-category, as in \cite[2.0.0.7]{lurie1}. Let $\mathbb{A}_\infty$ be the associative $\infty$-operad as in \cite[4.1.1.6]{lurie1}. 
Let $A$ be an $\mathbb{A}_\infty$-algebra in $\Cinf$, as in \cite[4.1.1.6]{lurie1}.  We denote $\mathcal{Mod}_A(\Cinf)$ the $\infty$-category of right $A$-modules, instead of $\mathcal{RMod}_A(\Cinf)$ as in \cite[4.2.1.13]{lurie1}. 
By an \emph{$\mathbb{A}_\infty$-coalgebra} in $\Cinf$ we mean an $\mathbb{A}_\infty$-algebra in the opposite category $\Cinf\op$. 
If $C$ is an $\mathbb{A}_\infty$-coalgebra in $\Cinf$, we define the category of right $C$-comodules in $\Cinf$ as:
\(
\comodinf_C(\Cinf):= {\left(\modinf_C(\Cinf\op) \right)}\op.
\)

Given a symmetric monoidal model category $\M$, any coalgebra $C$ in $\M$ that is cofibrant as an object in $\M$ can be regarded as an $\mathbb{A}_\infty$-coalgebra in the underlying $\infty$-category $\N(\M_c)[\W^{-1}]$ of $\M$. See more details in \cite{coalginDK}. Essentially, a strictly coassociative coalgebra is naturally coassociative up to coherent homotopies.  Let $\M^\otimes$ be the operator category of $\M$.  Let $C$ be a coalgebra in $\M$. It can be regarded as an $\mathbb{A}_\infty$-coalgebra in $\N(\M)$ and there is an equivalence of $\infty$-categories:
\(
\N(\comod_C(\M)) \simeq \comodinf_C(\N(\M)).
\)
Suppose now that $\M$ is a combinatorial symmetric monoidal model category in which each object is cofibrant.
By \cite[4.1.7.6]{lurie1}, the underlying $\infty$-category $\N(\M)\left[ \W^{-1} \right]$ is symmetric monoidal via the derived tensor product of $\M$. Moreover, the localization functor $\N(\M^\otimes) \rightarrow \N(\M)[ \W^{-1} ]^\otimes$ is symmetric monoidal. In particular, we obtain a functor:
\(
\comodinf_C(\N(\M)) \longrightarrow \comodinf_C \left(\N(\M)\left[ \W^{-1} \right]\right).
\)
By construction, the functor sends weak equivalences to equivalences, and thus by the universal property of the localization, we obtain a functor:
\[
\begin{tikzcd}
\gamma_C:\N(\comod_C(\M)) \left[ \W_C^{-1} \right] \ar{r} & \comodinf_{{C}}\left(\N(\M) \left[ \W^{-1} \right]\right).
\end{tikzcd}
\]
Here $\W_C$ denotes the class of weak equivalences between right $C$-comodules in $\M$. 

Objects in $\comodinf_{{C}}\left(\N(\M) \left[ \W^{-1} \right]\right)$ are $C$-comodules in $\M$ with a coaction that is coassociative only up to higher homotopies. Objects in $\N(\comod_C(\M)) \left[ \W_C^{-1} \right]$ are $C$-comodules in $\M$ with a coaction that is strictly coassociative. We have just shown that comodules with a strictly coassociative coaction are naturally endowed with a coaction that is coassociative up to coherent homotopies.

If the functor $\gamma_C$ above is an equivalence of $\infty$-categories, then this shows that any $C$-comodule in $\M$ with a homotopy coherent coassociative coaction is weakly equivalent to a $C$-comodule with a strictly coassociative coaction. We say that we can \emph{rigidify the coaction}. 

We apply our above discussion to $\M=\chfo$, the category of non-negative chain complexes over $\k$.  By \cite[7.1.3.10]{lurie1}, the underlying $\infty$-category $\N(\chfo)[\W^{-1}]$ is equivalent as a symmetric monoidal $\infty$-category to connective $H\k$-modules in spectra, which we denote $\DDk$. Here $H\k$ denotes the Eilenberg-Mac Lane spectrum of $\k$. In this situation, we show we can always rigidify the coaction of a comodule.

\begin{thm}\label{main thm}
If $C$ is a simply connected coalgebra in $\chfo$, then there is a natural equivalence of $\infty$-categories:
\[
\begin{tikzcd}
\gamma_C:\N\left( \comod_C(\chfo)\right) \left[ \W^{-1}_C \right]\ar{r}{\simeq} & \comodinf_C\left(\Dinf^{\geq 0}(\k)\right),
\end{tikzcd}
\]
where $\W_C$ denotes the class of $C$-colinear chain maps that are quasi-isomorphisms in $\chfo$.
\end{thm}

\begin{proof}
We shall use a (dual) version of the Barr-Beck-Lurie monadicity theorem \cite[4.7.3.16]{lurie1}. More precisely, we show that the forgetful functors $\N( \comod_C(\chfo)) [ \W^{-1}_C ]\rightarrow \N(\chfo)[\W^{-1}]\simeq \Dinf^{\geq 0}(\k)$ and $\comodinf_C\left(\Dinf^{\geq 0}(\k)\right)\rightarrow \Dinf^{\geq 0}(\k)$ are both comonadic over $\Dinf^{\geq 0}(\k)$ via the comonad $-\otimes C\colon \Dinf^{\geq 0}(\k)\rightarrow \Dinf^{\geq 0}(\k)$. 
By comonadic, we mean the dual version of monadic as in \cite[4.7.3.4]{lurie1}.
Since $C$ is automatically cofibrant in $\chfo$, we have that the functor $-\otimes C\colon \chfo \rightarrow \chfo$ is already left derived and thus $X\otimes C$ is equivalent to the derived tensor product of $X$ and $C$ in $\chfo$ for any chain complex $X$.
The forgetful functor $\comodinf_C\left(\Dinf^{\geq 0}(\k)\right)\rightarrow \Dinf^{\geq 0}(\k)$ is comonadic by the dual version of \cite[4.7.2.5]{lurie1}.

We are only left to check that the functor $\N( \comod_C) [ \W^{-1}_C ]\rightarrow \Dinf^{\geq 0}(\k)$ is comonadic. 
By \cite[4.7.3.5]{lurie1}, this is equivalent to show that it is conservative and preserves split coaugmented cosimplicial (in the dual sense of \cite[4.7.2.2]{lurie1}).
Since every object is cofibrant in $\comod_C$, the forgetful functor $U\colon \comod_C\rightarrow \chfo$ is already left derived. 
Since $U$ preserves and reflects weak equivalences by definition of the model structures, we obtain that its left derived version $\N( \comod_C) [ \W^{-1}_C ]\rightarrow \Dinf^{\geq 0}(\k)$ is conservative.
By \cite[1.3.4.16, 1.3.4.23, 1.3.4.25]{lurie1}, since $\comod_C$ is a combinatorial model category, homotopy limits of cosimplicial diagrams in $\comod_C$ are equivalent to limits in the $\infty$-category $\N( (\comod_C)_f) [ \W^{-1}_C ]$.
So let $M^{-1}\rightarrow M^{\bullet}$ be a $U$-split cosimplicial diagram of fibrant objects in $\comod_C$ and let us show that $M^{-1}$ is the homotopy limit $\holim_\Delta^C(M^\bullet)$ in $\comod_C$. The natural map $M^{-1}\rightarrow \holim_\Delta^C(M^\bullet)$ fits in a commutative diagram in $\comod_C$
\[
 \begin{tikzcd}
 M^{-1} \ar{r} \ar{d} & \holim_\Delta^C (M^\bullet) \ar{d} \\
 \holim_\Delta^C \Omega^\bullet(M^{-1}, C, C) \ar{r} & \holim_{\Delta \times \Delta}^C \Omega^\bullet(M^\bullet, C, C).
 \end{tikzcd}
 \]
The left  vertical map is a quasi-isomorphisms by Theorem \ref{thm: fibrant replacement as cobar}. The right vertical map is also a quasi-isomorphism by \cite[19.4.2]{hir} as it is a map of cosimplicial diagrams of fibrant objects (using that both $M^n$ and $\Omega(M^n,C,C)$ are fibrant by Theorem \ref{thm: fibrant replacement as cobar})  which is a quasi-isomorphism levelwise.
The bottom horizontal map is a quasi-isomorphism as $M^{-1}\rightarrow M^\bullet$ is $U$-split. Indeed, we have that $U(M^{-1})\rightarrow \holim_\Delta M^\bullet$ is a quasi-isomorphism in $\chfo$.
Since the cofree functor $-\otimes C\colon\chfo\rightarrow \comod_C$ preserves limits and quasi-isomorphisms, we obtain a quasi-isomorphism $U(M^{-1})\otimes C \rightarrow  \holim_\Delta(M^\bullet)\otimes C\simeq \holim_\Delta^C (U(M^\bullet)\otimes C)$ and $U(M^{-1})\otimes C \rightarrow U(M^\bullet)\otimes C$ remains $U$-split. 
Reiterating this argument shows that the bottom horizontal map is a quasi-isomorphism in $\comod_C$, thus we can conclude that the top horizontal arrow is a quasi-isomorphism.
\end{proof} 

\section{The derived cotensor product}\label{Section: derived cotensor}

We assume $C$ is a \emph{cocommutative} simply connected coalgebra in $\chfo$ throughout this section. The cotensor product over $C$ results in a symmetric monoidal structure $(\comod_C, \square_C, C)$. 
We wish to induce a similar symmetric monoidal structure on the $\infty$-category of $\comodinf_C(\Dinf^{\geq 0}(\k))$. We cannot reproduce the techniques on modules as in \cite[4.5.2.1]{lurie1} as the derived tensor product in $\DDk$ does not commute with totalizations. Instead, we use our previous rigidification result to derive the cotensor product on $\comod_C$. 

Given a symmetric monoidal model category $(\M, \otimes)$, we can Quillen left derive the tensor product $\otimes^\mathbb{L}$ to obtain a symmetric monoidal structure on the homotopy category of $\M$ (see \cite[4.3.2]{hovey}). It also endows its underlying $\infty$-category $\N(\M_c)\left[ \W^{-1}\right]$ with a symmetric monoidal structure (see \cite[4.1.7.6]{lurie1}). In the case of modules over a commutative algebra $A$, the derived tensor product $M\otimes^\mathbb{L} N$ of $A$-modules $M$ and $N$ is given by a two-sided bar construction $B(M, A, N)$ (see \cite[4.4.2.8]{lurie1}).  However, we come across an obstacle: the monoidal structure in $\comod_C$ is not a monoidal model structure. The bifunctor $-\ccotens - : \comod_C\times \comod_C \rightarrow \comod_C $ is not a Quillen bifunctor in the sense of \cite[4.2.1]{hovey}. It is not even part of a closed monoidal structure. 
It is not either a ``co-Quillen bifunctor" as $M\ccotens-:\comod_C\rightarrow \comod_C$ does not preserve fibrations, even if $M$ is fibrant.
Nevertheless, we can still derive the cotensor product using the following general method in $\infty$-categories from \cite[4.1.7.4]{lurie1}.

\begin{prop}
Let $(\M,\otimes, \I)$ be a symmetric monoidal category endowed with a model structure. Denote by $\M_f$ the full subcategory spanned by fibrant objects.
Suppose 
\begin{enumerate}[label=\upshape\textbf{(\roman*)}]
\item for any fibrant object $M$ in $\M$, the functors $M\otimes -:\M\rightarrow \M$ and $-\otimes M:\M\rightarrow \M$ preserve fibrant objects and weak equivalences between fibrant objects;
\item\label{item: comonoidal model cat} the monoidal unit $\I$ is fibrant.
\end{enumerate}
Then the Dwyer-Kan localization $\N(\M_f)[{\W}^{-1}]$ of $\M$ can be given the structure of a symmetric monoidal $\infty$-category via the symmetric monoidal Dwyer-Kan localization:
\(
\N(\M_f^\o) \rightarrow \N(\M_f)[{\W}^{-1}]^\o,
\)
where ${\W}$ is the class of weak equivalences restricted to fibrant objects in $\M$.
\end{prop}

Subsequently, we will show that the cotensor product of fibrant $C$-comodules is a fibrant $C$-comodule (Corollary \ref{cor: cotensor of fibrant is fibrant}) and that $M\square_C-\colon \comod_C\rightarrow \comod_C$ preserves quasi-isomorphisms if $M$ is fibrant (Corollary \ref{cor: coflat preserves weak equivalences}). 
This provides a derived cotensor product $\dcotens$ on $\comod_C$ and thus on $\N(\comod_C)[{\W}^{-1}_C]$ using fibrant replacement.
In fact, we can use the fibrant replacement $M\stackrel{\simeq}\hookrightarrow \Omega(M,C,C)$ of Theorem \ref{thm: fibrant replacement as cobar} to show that if $N$ is a fibrant $C$-comodule,  then $M\dcotens N\simeq \Omega(M,C,C)\square_C N$.
Although $-\square_C N\colon \comod_C\rightarrow \comod_C$ does not preserve infinite homotopy limits, it does preserves direct-sum totalization of double chain complex and thus we still obtain $\Omega(M,C,C)\square_C N\simeq \Omega(M,C,N)$.
Combining the results, we obtain the following.

\begin{thm}\label{thm: derived cotensor product}
The $\infty$-category $\comodinf_C(\Dinf^{\geq 0}(\k))$ of homotopy coherent comodules in $\Dinf^{\geq 0}(\k)$  over a simply connected cocommutative coalgebra $C$ in $\chfo$ is symmetric monoidal with respect to the derived cotensor product $\dcotens$ defined by $M\dcotens N\simeq \Omega(M,C,N)$ if $M$ and $N$ are $C$-comodules.
\end{thm}

\subsection{Coflat Comodules}

We show here, in Proposition \ref{prop: coflat=fibrant}, that fibrant $C$-comodules in the model category $\comod_C$ are precisely the \emph{coflat $C$-comodules}. In particular, this will show that the cotensor product of fibrant $C$-comodules remains fibrant (see Corollary \ref{cor: cotensor of fibrant is fibrant}).

If $M$ is a $C$-comodule, then the induced functor $M\square_C-\colon \comod_C\rightarrow \comod_C$ is left exact, preserves finite limits and filtered colimits. This is simply due to the fact that $M$ is automatically flat in $\chfo$ and thus $M\otimes -\colon \chfo\rightarrow \chfo$ preserves finite limits and monomorphisms, and that equalizers in $\chfo$ commutes with filtered colimits.
We say $M$ is a \emph{coflat }$C$-comodule if $M\square_C-\colon \comod_C\rightarrow \comod_C$ is (right) exact.

\begin{prop}
Let $M$ and $N$ be coflat $C$-comodules. Then $M\ccotens N$ is a coflat $C$-comodule.
\end{prop}

\begin{proof}
This follows from the fact that composition of right exact functors is a right exact functor and that $(M\square_C N) \square_C P\cong M \square_C (N\square_C P)$ for any $C$-comodule $P$.
\end{proof}

We will show below in Proposition \ref{prop: coflat=fibrant} that a $C$-comodule is fibrant if and only if it is coflat. Thus the above results immediately gives the following.

\begin{cor}\label{cor: cotensor of fibrant is fibrant}
Let $M$ and $N$ be fibrant $C$-comodules. Then $M\ccotens N$ is fibrant.
\end{cor}

We first need some preliminary results of homological algebra and model categories. The following observation is useful to identify certain fibrations in $\comod_C$.

\begin{lem}[Bousfield]\label{lem: Bousfield}
Let $\A$ be an abelian category endowed with a model structure where acyclic cofibrations are precisely monomorphisms with acyclic cokernels. Let $f\colon X\rightarrow Y$ be an epimorphism in $\A$. Let $F$ be its kernel. 
Then $f$ is a fibration if and only if $F$ is fibrant.
\end{lem}

\begin{proof}
A fibration always has fibrant kernel, regardless of being an epimorphism. 
This is because pullbacks preserve fibrations and the kernel $F$ is given by the pullback:
\[
\begin{tikzcd}
F \pull \ar[two heads]{d} \ar{r} & X \ar[two heads]{d}{f}\\
0 \ar{r} & Y.
\end{tikzcd}
\]
Now suppose $F$ is fibrant, let us show that $f$ is a fibration. Since $\A$ is a model category, we can factor $f=f'\circ i$, where $i$ is an acyclic cofibration and $f'$ is a fibration. Denote $X'$ the target of $i$ and $F'$ the kernel of $f'$.
We obtain the following morphism of short exact sequences in $\A$:
\[
\begin{tikzpicture}[baseline= (a).base]
\node[scale=1] (a) at (1,1){
\begin{tikzcd}
0 \ar{r} & F \ar{d} \ar{r} & X \ar{r}{f}\ar[hook]{d}{i}[swap]{\simeq} & Y \ar[equals]{d} \ar{r} & 0\\
0 \ar{r} & F' \ar{r} & X'\ar{r}{f'} & Y \ar{r} & 0.
\end{tikzcd}
};
\end{tikzpicture}
\]
We have used the fact that since $f$ is an epimorphism and $f=f'\circ i$, then $f'$ must also be an epimorphism.
Since $i$ is a monomorphism, the snake lemma guarantees that the induced map $F\rightarrow F'$ is also a monomorphism. 
Therefore we can take the cokernels of the vertical maps:
\[
\begin{tikzpicture}[baseline= (a).base]
\node[scale=1] (a) at (1,1){
\begin{tikzcd}
& 0\ar{d} & 0\ar{d} & 0 \ar{d} & \\
0 \ar{r} & F \ar[hook]{d} \ar{r} & X \ar{r}{f}\ar[hook]{d}{i}[swap]{\simeq} & Y \ar[equals]{d} \ar{r} & 0\\
0 \ar{r} & F' \ar{r} \ar{d} & X'\ar{r}{f'} \ar{d} & Y \ar{r}\ar{d} & 0\\
0 \ar{r} & K\ar{d} \ar{r} & K'\ar{d} \ar{r} & 0\ar{d}\ar{r} & 0.\\
& 0 & 0 & 0
\end{tikzcd}
};
\end{tikzpicture}
\]
The $9$-lemma guarantees that the third row is exact, and thus $K$ is acyclic. Therefore $F\rightarrow F'$ is an acyclic cofibration. 
Since $F$ is fibrant, then there is a section $\ell\colon F'\rightarrow F$.
We define $P$ to be the pushout of $\ell$ along $F'\hookrightarrow X'$ in $\A$.
In an abelian category, pushouts preserve monomorphisms so $F\rightarrow P$ is a monomorphism. Pushouts also preserve cokernels, thus $Y$ is the cokernel of $F\rightarrow P$. 
Therefore we obtain the following composite of short exact sequences:
\[
\begin{tikzpicture}[baseline= (a).base]
\node[scale=1] (a) at (1,1){
\begin{tikzcd}
0 \ar{r} & F \ar[hook]{d}[swap]{\simeq} \ar{r} & X \ar{r}{f}\ar[hook]{d}{i}[swap]{\simeq} & Y \ar[equals]{d} \ar{r} & 0\\
0 \ar{r} & F'\ar{d}{\ell} \ar{r} & X'\ar{r}{f'} \ar{d} & Y \ar{r}  \ar[equals]{d}& 0\\
0 \ar{r} & F \ar{r} & P \ar{r} & Y \ar{r} & 0.
\end{tikzcd}
};
\end{tikzpicture}
\]
The composite of the left vertical arrows is the identity on $F$ by construction of $\ell$. By the $5$-lemma, we get that $P$ is isomorphic to $X$. Therefore, we have just shown that $f$ is a retract of $f'$ which is a fibration. Hence $f$ is also a fibration.
\end{proof} 

 Let $M$ be a $C$-comodule. Define \(\cotor^i(M,-)\colon\comod_C\rightarrow \comod_C\) to be the $i$-th right derived functor of $M\ccotens -$, for $i\geq 0$, in the abelian categorical sense.
We have that $\cotor^0(M,N)=M\ccotens N$ for any comodules $M$ and $N$. If $N$ is an injective $C$-comodule, then $\cotor^i(M,N)=0$ for any comodule $M$ and $i>0$. It is immediate to verify the following.

\begin{prop}\label{prop: coflat equivalent conditions}
Let $M$ be a $C$-comodule. 
The following are equivalent:
\begin{enumerate}[label=\upshape\textbf{(\roman*)}]
\item the $C$-comodule $M$ is coflat;
\item the functor $M\ccotens-\colon \comod_C\rightarrow \comod_C$ preserves all colimits;
\item the functor $M\ccotens-\colon\comod_C\rightarrow \comod_C$ is a left adjoint;
\item for any $C$-comodule $N$, we have $\cotor^1(M, N)=0$.
\end{enumerate}
\end{prop}

The category $\comod_C$ is $\k$-linear, the hom set of $C$-colinear maps $M\rightarrow N$ can be endowed with a $\k$-module structure which we denote $\Hom_C(M,N)$. Equivalently, these are the $0$-cycles of the subchain complex of the internal hom $\underline{\Hom}(M,N)\in \chfo$ obtained as the kernel of $f\mapsto \rho'\circ f - (f\otimes \id_C)\circ \rho$, for $\rho$ and $\rho'$ the coactions of $M$ and $N$ respectively.
As usual, given $C$-comodules $M$ and $N$, we have that any short exact sequence in $\comod_C$
\[
\begin{tikzcd}
0\ar{r} & N \ar{r} & E \ar{r} & M \ar{r} & 0
\end{tikzcd}
\]
splits if and only if $\Ext^i(M,N)=0$.
\begin{prop}\label{prop: coflat=fibrant}
A $C$-comodule is coflat if and only if it is fibrant in the model category $\comod_C$.
\end{prop}

\begin{proof}
As $C$ is simply connected, the split $C=\k\oplus \overline{C}$ gives $\k$ a $C$-comodule structure. Thus any chain complex $X$ can be given a trivial $C$-comodule structure $X\cong X\o \k \longrightarrow X\o C$.
For $n\geq 0$, let $S^n$ be the chain complex $\k$ concentrated in degree $n$ and zero elsewhere. 
For $n\geq 1$, let $D^n$ be the acyclic  chain complex that is $\k$ in degrees $n-1$ and $n$ with identity as differential, and zero elsewhere.
We give $S^n$ and $D^n$ the trivial $C$-comodule structure.
Let $F$ be a $C$-comodule. We show that the following statements are equivalent.
\begin{enumerate}[label=\upshape\textbf{(\roman*)}]
\item\label{item: coflat=fibrant 1} $F$ is a coflat $C$-comodule.
\item\label{item: coflat=fibrant 2} $\cotor^1(S^n, F)=0$,  for all $n\geq 0$.
\item\label{item: coflat=fibrant 3} $\cotor^1(S^0, F) =0$.
\item\label{item: coflat=fibrant 4} $\mathsf{Ext}_C^1(D^n, F)=0$,  for all $n\geq 1$.
\item\label{item: coflat=fibrant 5} $\mathsf{Ext}^1_C(M, F)=0$ for any acyclic $C$-comodule $M$.
\item\label{item: coflat=fibrant 6} $F$ is a fibrant $C$-comodule.
\end{enumerate}

\ref{item: coflat=fibrant 1} $\Leftrightarrow$ \ref{item: coflat=fibrant 2} 
We only need to show \ref{item: coflat=fibrant 2} $\Rightarrow$ \ref{item: coflat=fibrant 1}.
Suppose $\cotor^1(S^n, F)=0$ for all $n\geq 0$.
As cotensor product preserves coproducts, then $\cotor^1(\bigoplus_\lambda S^n, F)=0$ for all $n\geq 0$ and any ordinal $\lambda$. 
If $V$ is $\k$-module, it is a retract of free $\k$-module as it is automatically projective, and thus if we view $V$ as a chain complex concentrated in degree zero, then $S^n\otimes V$ is a retract of $\bigoplus_\lambda S^n$ for some $\lambda$.
Thus $\cotor^1(S^n\otimes V, F)=0$.

We now introduce a notation. For any  chain complex $M$, and for $n\geq 0$, define $M_{\leq n}$ as the brutal truncation of $M$ \cite[1.2.7]{weibel} where $(M_{\leq n})_i =M_i$ for $i\leq n$ and zero otherwise.
As $(M_{\leq n} \otimes C)_i=\left(M\otimes C_{}\right)_i$ for $0\leq i\leq n$, we obtain an unique $C$-comodule structure on $M_{\leq n}$ from a $C$-comodule structure on $M$ such that the inclusion $M_{\leq n}\hookrightarrow M$ is $C$-colinear.
In particular, any $C$-comodule $M$ is the filtered colimit $\colim_n M_{\leq n}$ in $\comod_C$.

Let us prove that $\cotor^1(M,F)=0$ for any $C$-comodule $M$.
We first prove by induction on $n\geq 0$ that $\cotor^1(M_{\leq n}, F)=0$. 
For the initial case, notice that $M_{\leq 0}\cong S^0\otimes V$ for some $\k$-module $V$, and thus $\cotor^1(M_{\leq 0}, F)=0$ by our above argument.
Now suppose $\cotor^1(M_{\leq n}, F)=0$ for some $n\geq 0$. Then we obtain a short exact sequence of $C$-comodules:
\[
\begin{tikzcd}
0\ar{r} & M_{\leq n} \ar{r} & M_{\leq n+1} \ar{r} & S^{n+1}\otimes V\ar{r} & 0,
\end{tikzcd}
\]
for some $\k$-module $V$.
The induced long exact sequence from right deriving the cotensor $-\ccotens F$ in particular gives the exact sequence:
\[
\begin{tikzcd}[column sep= small]
\cotor^1(M_{\leq n}, F) \ar{r} & \cotor^1(M_{\leq n+1}, F) \ar{r} & \cotor^1(S^{n+1}\otimes V, F).
\end{tikzcd}
\]
Thus by induction, and our above argument, we get $\cotor^1(M_{\leq n+1}, F)=0$.
As cotensor product commutes with filtered colimits, we get:
\(
\cotor^1(M, F) \cong \colim_n \cotor^1(M_{\leq n}, F),
\)
and hence we can conclude $\cotor^1(M, F)=0$, for all $C$-comodules $M$. Thus $F$ is coflat.

\ref{item: coflat=fibrant 2} $\Leftrightarrow$ \ref{item: coflat=fibrant 3}
We only need to show that \ref{item: coflat=fibrant 3} $\Rightarrow$ \ref{item: coflat=fibrant 2}. For all $n\geq 0$, notice that:
\(
S^n\ccotens F \cong S^n \otimes (S^0 \ccotens F).
\)
Since the functor $S^n\otimes-:\comod_C\rightarrow \comod_C$ is exact, the result follows.

\ref{item: coflat=fibrant 3} $\Leftrightarrow$ \ref{item: coflat=fibrant 4}
For any $C$-colinear map $I\rightarrow K$ the induced chain map $\Homc(D^n, I)\rightarrow \Homc( D^n, K)$ is an epimorphism for all $n\geq 1$ if and only if $S^0\ccotens I\rightarrow S^0 \ccotens K$ is an epimorphism of $C$-comodules.
Indeed, this follows from the isomorphism
\(
\Hom_C(D^n, M)\cong (S^0 \square_C M)_{n}
\)
for any $C$-comodule $M$.

Now, let $I$ be an injective $C$-comodule such that $F$ maps into $I$. Then we can form the short exact sequence of $C$-comodules:
\[
\begin{tikzcd}
0 \ar{r} & F \ar{r} & I \ar{r} & K \ar{r} & 0.
\end{tikzcd}
\]
Thus $\cotor^1(S^0, F)=0$ if and only if $S^0\ccotens I \rightarrow S^0 \ccotens K$ is an epimorphism, and $\Ext^1(D^n, F)=0$ if and only if $\Homc(D^n, I)\rightarrow \Homc( D^n, K)$ is an epimorphism for all $n\geq 1$. We can conclude by our above argument.

\ref{item: coflat=fibrant 4} $\Leftrightarrow$ \ref{item: coflat=fibrant 5}
We only need to show that \ref{item: coflat=fibrant 4} $\Rightarrow$ \ref{item: coflat=fibrant 5}.
Notice that $\Ext^1(D^n, F)=0$ for all $n\geq 1$, implies that $\Ext^1(D^n\otimes V, F)=0$ for any $\k$-module $V$ and all $n\geq 1$.
Indeed, since $V$ is a retract of a free $\k$-module, we only need to show that $\Ext^1(\bigoplus_\lambda D^n, F)=0$ for some ordinal $\lambda$. But this follows from:
\[
\Homc\left(\bigoplus_\lambda D^n, F\right)\cong\prod_\lambda \Homc (D^n, F).
\]
If $M$ is an acyclic chain complex, and $n\geq 1$, let $M_{<n}$ be the good truncation of $M$ \cite[1.2.7]{weibel}, where  $(M_{<n})_i=M_i$ for $0\leq i <n$, while $(M_{<n})_n$ is a choice of $\k$-module $V_n$ in $M_n$ that is isomorphic to $M_n/Z_n(M)$, and $(M_{<n})_i=0$ otherwise.
Given an acyclic $C$-comodule $M$, there is a unique $C$-comodule structure on $M_{<n}$ such that $M_{<n}\hookrightarrow M$ is $C$-colinear.
Indeed, we need to see that the coaction $\rho\colon M\rightarrow M\otimes C$ restricts to $\rho_{<n}\colon M_{<n}\rightarrow M_{<n}\otimes C$. 
For $i<n$, as $(M_{<n}\otimes C)_i=(M\otimes C)_i$, we see directly that $\rho$ restricts to $M_{<n}$ on the $i$-th level. We now need to verify that $\rho_n\colon M_n\rightarrow (M\otimes C)_n$ sends $V_n=(M_{<n})_n$ to $(M_{<n}\otimes C)_n$.
Notice first that:
\[
(M_{<n}\otimes C)_n=(V_n\otimes \k) \oplus \bigoplus_{\substack{a+b=n \\ a>0}} M_a\otimes C_b \, \subseteq \, (M_n\otimes \k)\oplus \bigoplus_{\substack{a+b=n \\ a>0}} M_a\otimes C_b = (M\otimes C)_n
\]
But by counitality of the coaction, we know that if $m\in M_n$, then the value of $\rho_n(m)$ on the summand $M_n\otimes \k$ in $(M\otimes C)_n$ must be $m\otimes 1$. Thus if $m\in V_n$, then $\rho_n(m)\in (M_{<n}\otimes C)_n$. This shows $M_{<n}$ is a subcomodule of $M$. 
Therefore $M$ is the filtered colimit $\colim_n M_{<n}$ in $\comod_C$.

Let us now show that $\Ext^1(M, F)=0$ for any acyclic $C$-comodule $M$. 
We first prove by induction on $n\geq 1$ that $\Ext^1(M_{<n}, F)=0$.
For the initial case, notice that $M_{<1}\cong D^1\otimes V$ for some $\k$-module $V$ and thus $\Ext^1(M_{<1}, F)=0$ by our above argument.
Now suppose $\Ext^1(M_{<n}, F)=0$ for some $n\geq 1$. 
Then we obtain a short exact sequence of $C$-comodules for some $\k$-module $V$:
\[
\begin{tikzcd}[column sep=small]
0 \ar{r} & M_{<n} \ar{r} & M_{< n+1} \ar{r} & D^{n+1}\otimes V \ar{r} & 0.
\end{tikzcd}
\]
The induced long exact sequence from left deriving the functor $\Homc(-,F)$ in particular gives the exact sequence:
\[
\begin{tikzcd}[column sep= small]
\Ext^1(D^{n+1}\otimes V, F)\ar{r} & \Ext^1(M_{< n+1}, F) \ar{r} & \Ext^1(M_{< n}, F)
\end{tikzcd}
\]
Thus by induction, and our above argument, we get $\Ext^1(M_{< n+1}, F)=0$. 

The tower $\{\Hom_C(M_{n+1}, F)\rightarrow \Hom_C(M_{<n}, F)\}$ satisfies the Mittag-Leffler conditions \cite[3.5.6]{weibel}, hence
$\Ext^1(M, F)\cong \lim_n \Ext^1(M_{<n}, F)$ as in \cite[3.5.10]{weibel}.
Therefore we can conclude $\Ext^1(M, F)=0$ for any acyclic $C$-comodule $M$.

\ref{item: coflat=fibrant 5} $\Leftrightarrow$ \ref{item: coflat=fibrant 6}
Let us first show \ref{item: coflat=fibrant 5} $\Rightarrow$ \ref{item: coflat=fibrant 6}. 
Suppose $\Ext^1(M, F)=0$ for any acyclic $C$-comodule $M$.
Factor the map $F\rightarrow 0$ in $\comod_C$ as
\[
\begin{tikzcd}
F \ar{rr} \ar[hook]{dr}[swap]{\simeq} & & 0\\
& E\ar[two heads]{ur} & 
\end{tikzcd}
\]
such that $E$ is fibrant.  Let $K$ be the cokernel of $F\hookrightarrow E$. Then we obtain a short exact sequence in $\comod_C$
\[
\begin{tikzcd}
0 \ar{r} & F \ar{r}{\simeq} & E \ar{r} & K \ar{r} & 0.
\end{tikzcd}
\]
Thus $K$ is an acyclic $C$-comodule. But since $\Ext^1(K, F)=0$, then the short exact sequence must split.
Thus $F$ is a retract of $E$, and hence $F$ is fibrant.

Let us show now \ref{item: coflat=fibrant 6} $\Rightarrow$ \ref{item: coflat=fibrant 5}.
Suppose we are given any extension of $F$ with an acyclic $C$-comodule $M$:
\[
\begin{tikzcd}
0 \ar{r} & F \ar{r} & E \ar{r} & M \ar{r} & 0.
\end{tikzcd}
\]
Then by Lemma \ref{lem: Bousfield}, we get that the map $E\rightarrow M$ must be a fibration of $C$-comodules.
But the lifting property provides the dashed map in the following commutative diagram in $\comod_C$
\[
\begin{tikzcd}
0 \ar[hook]{d}[swap]{\simeq} \ar{r} & E \ar[two heads]{d}\\
M \ar[equals]{r} \ar[dashed]{ur} & M.
\end{tikzcd}
\]
Thus the extension must split. 
Hence $\Ext^1(M, F)=0$. \qedhere
\end{proof}

\subsection{An Eilenberg-Moore Spectral Sequence} We show here that if $M$ is a fibrant $C$-comodule (i.e. coflat), then the functor $M\ccotens-:\comod_C\rightarrow \comod_C$ preserves quasi-isomorphisms. We shall use classical methods from \cite{EMSS}.
As noticed in \cite[A1.2.12]{ravenel}, it is useful to observe that for $i\geq 0$ and all $C$-comodules $M$ and $N$, we have an isomorphism of $C$-comodules:
\(\cotor^i(M, N)\cong H^i(\overline{\Omega}^\bullet({M}, C, {N})).\)
Recall that to any chain complex $X$, we can regard its homology $H_*(X)$ as a chain complex with trivial differentials. As $C$ is a simply connected coalgebra, then $H_*(C)$ is also a simply connected coalgebra (with trivial differentials). Moreover, for any $C$-comodule $M$, we can check that $H_*(M)$ is a $H_*(C)$-comodule.

\begin{thm}[Eilenberg-Moore Spectral Sequence]\label{thm: EMSS}
Let $M$ be a fibrant $C$-comodule, i.e. a coflat $C$-comodule.
Let $N$ be any $C$-comodule.
Then there is a convergent spectral sequence:  
\[E^2_{\bullet,q}= \mathsf{CoTor}^q_{H_*(C)}(H_*(M), H_*(N)) \Rightarrow E^\infty_{\bullet, 0}=H_*(M\ccotens N).
\] 
\end{thm}

\begin{proof}
As previously observed, the second quadrant double chain complex $\overline{\Omega}^\bullet(M,C,N)_\bullet$ of Definition \ref{def: double complex cobar} is bounded since $C$ is simply connected.
Thus the two associated spectral sequences to the double complex converge, see \cite[2.15]{mccleary}.
The first spectral sequence has its $E^1$-page induced by the cohomology of the columns, and therefore:
\(
E^1_{\bullet,q}=H^q(\overline{\Omega}^\bullet(M,C,N))\cong \cotor^q(M,N).
\)
Since $M$ is a coflat $C$-comodule, then $E^1_{\bullet,q}=0$ for all $q\geq 1$, and we have $E^1_{\bullet, 0}= M\ccotens N$. Thus the spectral sequence collapses onto its second page $E^2_{\bullet, 0}=H_*( M\ccotens N)$.
The second spectral sequence has its $E^1$-page induced by the homology of the rows, and therefore:
\(
E^1_{\bullet,q}=H_*(\overline{\Omega}^q(M,C,N)))=\overline{\Omega}^q(H_*(M), H_*(C), H_*(N)).
\)
Thus, as its $E^2$-page is given by the cohomology of the induced cochain complex, we obtain:
\(
E^2_{\bullet, q}= \mathsf{CoTor}^q_{H_*(C)}(H_*(M), H_*(N)).
\)
It converges to the page with trivial rows except its $0$-th column which is given by the homology $H_*(M\ccotens N)$.
\end{proof}

\begin{cor}\label{cor: coflat preserves weak equivalences}
Let $C$ be a simply connected coalgebra in $\chfo$.
Let $M$ be a fibrant $C$-comodule. (i.e. a coflat $C$-comodule). Then the functor $M\ccotens-:\comod_C\rightarrow \comod_C$ preserves quasi-isomorphisms.
\end{cor}

\begin{proof}
Let $Y\stackrel{\simeq}\longrightarrow Y'$ be a quasi-isomorphism of $C$-comodules. 
It induces an isomorphism $H_*(Y)\cong H_*(Y')$ of $H_*(C)$-comodules.
Therefore we obtain:
\[
\mathsf{CoTor}^q_{H_*(C)}(H_*(X), H_*(Y)) \cong \mathsf{CoTor}^q_{H_*(C)}(H_*(X), H_*(Y')),
\]
for all $q\geq 0$. By Theorem \ref{thm: EMSS}, we obtain $H_*(X\ccotens Y)\cong H_*(X\ccotens Y')$ via the map $Y\rightarrow Y'$. Then $X\ccotens Y \stackrel{\simeq}\longrightarrow X \ccotens Y'$ is a quasi-isomorphism of $C$-comodules.
\end{proof}

\subsection{Change of Coalgebras}
We observe here a direct consequence from Corollary \ref{cor: coflat preserves weak equivalences}.
Let $f\colon C\rightarrow D$ be a map of simply connected cocommutative coalgebras in $\chfo$. The map endows the coalgebra $C$ with a right $D$-comodule structure:
\(
\begin{tikzcd}
C\ar{r}{\Delta_C} & C\o C \ar{r}{\id_C \o f} & C\o D,
\end{tikzcd}
\)
such that $f\colon C\rightarrow D$ is a map of $D$-comodules.
We obtain a functor \(f^*\colon\comod_C\rightarrow \comod_D,\) where each  $C$-comodule $(M, \rho)$ is sent to the $D$-comodule $(M, (\id_M \o f)\circ \rho)$. We shall often write $f^*(M)$ simply as $M$.

Given any $D$-comodule $M$, we can form the cotensor of $D$-comodules $M\cotens_D C$, which can be endowed with the structure of $C$-comodule $M \cotens_D C\rightarrow M\square_D (C\otimes C)\cong (M\square_D C)\otimes C$  via the comultiplication on $C$. 
Therefore, we obtain a functor $-\cotens_D C\colon \comod_D\rightarrow \comod_C$ which is right adjoint to $f^*$. 

\begin{prop}\label{prop: change of coalgebras}
Let $f\colon C\rightarrow D$ be a map of simply connected cocommutative coalgebras in $\chfo$.
Then the adjunction
\[
\begin{tikzcd}
 \comod_C \ar[shift left=2]{r}{f^*}[swap]{\perp} & \comod_D \ar[shift left=2]{l}{-\cotens_D C}
\end{tikzcd}
\]
is a Quillen pair. The adjunction is a Quillen equivalence if and only if the map $f$ is a quasi-isomorphism.
Moreover, when $f$ is a quasi-isomorphism, we obtain an equivalence of symmetric monoidal $\infty$-categories
\(
\comodinf_C(\Dinf^{\geq 0}(\k)) \simeq \comodinf_D(\Dinf^{\geq 0}(\k)),
\)
with respect to their derived cotensor product.
\end{prop}

\begin{proof}
The first statement follows directly from the fact that the functor $f^*$ preserves monomorphisms and weak equivalences. 
For the second statement, we shall apply \cite[1.3.16]{hovey}. Notice that $f^*$ reflects weak equivalences. Suppose first that $f$ is a weak equivalence. 
Now let $M$ be any fibrant $D$-comodule, the counit of the adjunction $M\cotens_D C \xrightarrow{\simeq} M\cotens_D D \cong M$
is a weak equivalence by Corollary \ref{cor: coflat preserves weak equivalences}. Conversely, if we suppose the adjunction to be a Quillen equivalence, then the map
\(
f\colon C\cong D\cotens_D C \longrightarrow D\cotens_D D\cong D
\)
must be a weak equivalence, as $D$ is always fibrant as a $D$-comodule.

The third statement, it follows from the universal property of the symmetric Dwyer-Kan localization \cite[4.1.7.4]{lurie1}. 
The functor $f^*\colon\comod_C \rightarrow \comod_D$ is colax monoidal with respect to the cotensor products. 
The induced lax monoidal map on the right adjoint:
\[
(M\cotens_D C) \ccotens (N\cotens_D C) \stackrel{\cong}\rightarrow (M\cotens_D N)\cotens_D C,
\]
is actually strong monoidal by associativity of the cotensor product, for any $D$-comodules $M$ and $N$. Thus we get a symmetric monoidal functor of $\infty$-categories:
\(
-\cotens_D C\colon\N((\comod_D)_f^\otimes) \longrightarrow \N((\comod_C)_f^\otimes).
\)
If we suppose $f$ to be a weak equivalence, then if we post-compose the functor with the symmetric monoidal Dwyer-Kan localization:
\(
 \N((\comod_C)_f^\otimes)\longrightarrow  \N((\comod_C)_f^\otimes)[\W^{-1}_\comod]\simeq \comodinf_C(\Dinf^{\geq 0}(\k)),
\)
which is strong monoidal, it sends weak equivalence in $\comod_D$ to equivalences in $\comodinf_C(\Dinf^{\geq 0}(\k))$.
Thus we obtain the desired equivalence of symmetric monoidal $\infty$-categories.
\end{proof}

\renewcommand{\bibname}{References}
\bibliographystyle{amsalpha}
\bibliography{biblio}

\newcommand{\etalchar}[1]{$^{#1}$}
\providecommand{\bysame}{\leavevmode\hbox to3em{\hrulefill}\thinspace}
\providecommand{\MR}{\relax\ifhmode\unskip\space\fi MR }
\providecommand{\MRhref}[2]{%
  \href{http://www.ams.org/mathscinet-getitem?mr=#1}{#2}
}
\providecommand{\href}[2]{#2}
\begin{thebibliography}{EKMM97}

\bibitem[AR94]{Adamek-Rosicky}
Ji\v{r}\'\i{} Ad\'amek and Ji\v{r}\'\i{} Rosick\'y, \emph{Locally presentable
  and accessible categories}, London Mathematical Society Lecture Note Series,
  vol. 189, Cambridge University Press, Cambridge, 1994. \MR{1294136}

\bibitem[BHK{\etalchar{+}}15]{left1}
Marzieh Bayeh, Kathryn Hess, Varvara Karpova, Magdalena K\c{e}dziorek, Emily
  Riehl, and Brooke Shipley, \emph{Left-induced model structures and diagram
  categories}, Women in topology: collaborations in homotopy theory, Contemp.
  Math., vol. 641, Amer. Math. Soc., Providence, RI, 2015, pp.~49--81.
  \MR{3380069}

\bibitem[{Bun}12]{ulrich}
Ulrich {Bunke}, \emph{{Differential cohomology}}, arXiv e-prints (2012),
  arXiv:1208.3961.

\bibitem[EKMM97]{EKMM}
Anthony Elmendorf, Igor Kriz, Michael~A. Mandell, and Peter May, \emph{Rings,
  modules, and algebras in stable homotopy theory}, Mathematical Surveys and
  Monographs, vol.~47, American Mathematical Society, Providence, RI, 1997,
  With an appendix by M. Cole. \MR{1417719}

\bibitem[EM66]{EMSS}
Samuel Eilenberg and John~C. Moore, \emph{Homology and fibrations. {I}.
  {C}oalgebras, cotensor product and its derived functors}, Comment. Math.
  Helv. \textbf{40} (1966), 199--236. \MR{203730}

\bibitem[FG12]{chiral}
John Francis and Dennis Gaitsgory, \emph{Chiral {K}oszul duality}, Selecta
  Math. (N.S.) \textbf{18} (2012), no.~1, 27--87. \MR{2891861}

\bibitem[GKR20]{left3}
Richard {Garner}, Magdalena {K\c{e}dziorek}, and Emily {Riehl}, \emph{Lifting
  accessible model structures}, J. Topol. \textbf{13} (2020), no.~1, 59--76.
  \MR{3999672}

\bibitem[Hin15]{hinrectification}
Vladimir Hinich, \emph{Rectification of algebras and modules}, Doc. Math.
  \textbf{20} (2015), 879--926. \MR{3404213}

\bibitem[Hir03]{hir}
Philip~S. Hirschhorn, \emph{Model categories and their localizations},
  Mathematical Surveys and Monographs, vol.~99, American Mathematical Society,
  Providence, RI, 2003. \MR{1944041}

\bibitem[HKRS17]{left2}
Kathryn Hess, Magdalena K\c{e}dziorek, Emily Riehl, and Brooke Shipley, \emph{A
  necessary and sufficient condition for induced model structures}, J. Topol.
  \textbf{10} (2017), no.~2, 324--369. \MR{3653314}

\bibitem[Hov99]{hovey}
Mark Hovey, \emph{Model categories}, Mathematical Surveys and Monographs,
  vol.~63, American Mathematical Society, Providence, RI, 1999. \MR{1650134}

\bibitem[HSS00]{SS}
Mark Hovey, Brooke Shipley, and Jeff Smith, \emph{Symmetric {S}pectra}, J.
  Amer. Math. Soc. \textbf{13} (2000), no.~1, 149--208. \MR{1695653}

\bibitem[KP23]{cothhshadow}
Sarah Klanderman and Maximilien P\'{e}roux, \emph{{Trace methods for
  coHochschild homology}}, arXiv e-prints (2023), arXiv:2301.11346.

\bibitem[KSV97]{corecognition}
J.~Klein, R.~Schw\"{a}nzl, and R.~M. Vogt, \emph{Comultiplication and
  suspension}, Topology Appl. \textbf{77} (1997), no.~1, 1--18. \MR{1443424}

\bibitem[Lur09]{htt}
Jacob Lurie, \emph{Higher topos theory}, Annals of Mathematics Studies, vol.
  170, Princeton University Press, Princeton, NJ, 2009. \MR{2522659}

\bibitem[Lur17]{lurie1}
Jacob Lurie, \emph{Higher algebra},
  \url{https://www.math.ias.edu/~lurie/papers/HA.pdf}, 2017, electronic book.

\bibitem[Man01]{mandello}
Michael~A. Mandell, \emph{{$E_\infty$} algebras and {$p$}-adic homotopy
  theory}, Topology \textbf{40} (2001), no.~1, 43--94. \MR{1791268}

\bibitem[May72]{geom}
J.~P. May, \emph{The geometry of iterated loop spaces}, Lecture Notes in
  Mathematics, Vol. 271, Springer-Verlag, Berlin-New York, 1972. \MR{0420610}

\bibitem[McC01]{mccleary}
John McCleary, \emph{A user's guide to spectral sequences}, second ed.,
  Cambridge Studies in Advanced Mathematics, vol.~58, Cambridge University
  Press, Cambridge, 2001. \MR{1793722}

\bibitem[MM02]{MM}
Michael~A. Mandell and Peter May, \emph{Equivariant orthogonal spectra and
  {$S$}-modules}, Mem. Amer. Math. Soc. \textbf{159} (2002), no.~755, x+108.
  \MR{1922205}

\bibitem[MMSS01]{MMSS}
Michael~A. Mandell, Peter May, Stefan Schwede, and Brooke Shipley, \emph{Model
  categories of diagram spectra}, Proc. London Math. Soc. (3) \textbf{82}
  (2001), no.~2, 441--512. \MR{1806878}

\bibitem[NS17]{NSpres}
Thomas Nikolaus and Steffen Sagave, \emph{Presentably symmetric monoidal
  {$\infty$}-categories are represented by symmetric monoidal model
  categories}, Algebr. Geom. Topol. \textbf{17} (2017), no.~5, 3189--3212.
  \MR{3704256}

\bibitem[P{\'e}r20]{phd}
Maximilien P{\'e}roux, \emph{Highly {S}tructured {C}oalgebras and {C}omodules},
  ProQuest LLC, Ann Arbor, MI, 2020, Thesis (Ph.D.)--University of Illinois at
  Chicago. \MR{4257381}

\bibitem[P{\'e}r22a]{coalgenr}
\bysame, \emph{The coalgebraic enrichment of algebras in higher categories}, J.
  Pure Appl. Algebra \textbf{226} (2022), no.~3, Paper No. 106849, 11.
  \MR{4291529}

\bibitem[P{\'e}r22b]{coalginDK}
\bysame, \emph{Coalgebras in the {D}wyer-{K}an localization of a model
  category}, Proc. Amer. Math. Soc. \textbf{150} (2022), no.~10, 4173--4190.
  \MR{4470166}

\bibitem[P{\'e}r24]{postnikovperoux}
\bysame, \emph{A monoidal {D}old-{K}an correspondence for comodules}, J. Pure
  Appl. Algebra \textbf{228} (2024), no.~8, Paper No. 107660. \MR{4720531}

\bibitem[Pos11]{leo}
Leonid Positselski, \emph{Two kinds of derived categories, {K}oszul duality,
  and comodule-contramodule correspondence}, Mem. Amer. Math. Soc. \textbf{212}
  (2011), no.~996, vi+133. \MR{2830562}

\bibitem[PS19]{perouxshipley}
Maximilien P\'{e}roux and Brooke Shipley, \emph{Coalgebras in symmetric
  monoidal categories of spectra}, Homology Homotopy Appl. \textbf{21} (2019),
  no.~1, 1--18. \MR{3852287}

\bibitem[Qui67]{quillen}
Daniel Quillen, \emph{Homotopical algebra}, Lecture Notes in Mathematics, No.
  43, Springer-Verlag, Berlin-New York, 1967. \MR{0223432}

\bibitem[Rav86]{ravenel}
Douglas~C. Ravenel, \emph{Complex cobordism and stable homotopy groups of
  spheres}, Pure and Applied Mathematics, vol. 121, Academic Press, Inc.,
  Orlando, FL, 1986. \MR{860042}

\bibitem[Shi04]{convenient}
Brooke Shipley, \emph{A convenient model category for commutative ring
  spectra}, Homotopy theory: relations with algebraic geometry, group
  cohomology, and algebraic {$K$}-theory, Contemp. Math., vol. 346, Amer. Math.
  Soc., Providence, RI, 2004, pp.~473--483. \MR{2066511}

\bibitem[SS03]{monmodSS}
Stefan Schwede and Brooke Shipley, \emph{Equivalences of monoidal model
  categories}, Algebr. Geom. Topol. \textbf{3} (2003), 287--334. \MR{1997322}

\bibitem[Tor20]{cobarr}
Takeshi Torii, \emph{On quasi-categories of comodules and {L}andweber
  exactness}, Bousfield classes and {O}hkawa's theorem, Springer Proc. Math.
  Stat., vol. 309, Springer, Singapore, 2020, pp.~325--380. \MR{4100661}

\bibitem[Wei94]{weibel}
Charles~A. Weibel, \emph{An introduction to homological algebra}, Cambridge
  Studies in Advanced Mathematics, vol.~38, Cambridge University Press,
  Cambridge, 1994. \MR{1269324}

\bibitem[Yua23]{yuan}
Allen Yuan, \emph{Integral models for spaces via the higher {F}robenius}, J.
  Amer. Math. Soc. \textbf{36} (2023), no.~1, 107--175. \MR{4495840}

\end{thebibliography}
\end{document}